\documentclass[pdflatex,sn-mathphys]{sn-jnl}

\jyear{2021}%

\theoremstyle{thmstyleone}%
\newtheorem{theorem}{Theorem}[section]
\newtheorem{proposition}[theorem]{Proposition}%
\newtheorem{lemma}[theorem]{Lemma}%
\newtheorem{corollary}[theorem]{Corollary}

\theoremstyle{thmstyletwo}%
\newtheorem{remark}{Remark}%

\theoremstyle{thmstylethree}%
\newtheorem{definition}{Definition}%

\usepackage{mathtools}
\usepackage{upgreek}
\DeclarePairedDelimiter{\abs}{\lvert}{\rvert}
\newcommand{\domD}{\mathscr{D}}
\newcommand{\dd}{\mathrm{d}}
\newcommand{\LC}{\mathbf{L}^{\mathrm{DE}}}
\renewcommand{\Im}{\operatorname{Im}}
\renewcommand{\Re}{\operatorname{Re}}
\renewcommand{\pi}{\uppi}
\DeclareMathOperator{\sinc}{sinc}
\DeclareMathOperator{\arsinh}{arsinh}
\DeclareMathOperator{\OO}{O}
\DeclareMathOperator{\ee}{e}
\numberwithin{equation}{section}
\begin{document}

\title[Improvement of selection formulas of mesh size and truncation numbers]{Improvement of selection formulas of mesh size and truncation numbers for the DE-Sinc approximation and its theoretical error bound}


\author*[1]{\fnm{Tomoaki} \sur{Okayama}}\email{okayama@hiroshima-cu.ac.jp}
\author[2]{\fnm{Shota} \sur{Ogawa}}

\affil*[1]{\orgdiv{Graduate School of Information Sciences},
\orgname{Hiroshima City University}, \orgaddress{\street{3-4-1,
Ozuka-higashi}, \city{Asaminami-ku}, \postcode{731-3194},
\state{Hiroshima}, \country{Japan}}}

\affil[2]{\orgdiv{} \orgname{Systemers Inc.},
\orgaddress{\street{Kurotaki Building 3F, 3-9-7, Shimbashi},
\city{Minato-ku}, \postcode{105-0004}, \state{Tokyo}, \country{Japan}}}


\abstract{
The Sinc approximation applied to
double-exponentially decaying functions
is referred to as the DE-Sinc approximation.
Because of its high efficiency,
this method has been used in various applications.
In the Sinc approximation, the mesh size and truncation numbers should be
optimally selected to achieve its best performance. However,
the standard selection formula has only been ``near-optimally'' selected
because the optimal formula of the mesh size cannot be expressed
in terms of elementary functions of truncation numbers.
In this study, we propose two improved selection formulas.
The first one is based on the concept by an earlier research
that resulted in a better selection formula
for the double-exponential formula.
The formula performs slightly better than the standard one,
but is still not optimal.
As a second selection formula,
we introduce a new parameter to propose truly optimal selection formula.
We provide explicit error bounds for both selection formulas.
Numerical comparisons show that the first formula gives
a better error bound than the standard formula,
and the second formula gives a much better error bound
than the standard and first formulas.
}




\maketitle

\section{Introduction and summary}\label{sec1}

The Sinc approximation is based on Shannon's sampling formula
\begin{equation}
 F(x)\approx \sum_{k=-\infty}^{\infty}
 F(kh) \sinc\left(\frac{x-kh}{h}\right),
\label{eq:sinc-approx-inf}
\end{equation}
where $\sinc x=\sin(\pi x)/(\pi x)$, and $h$ represents a mesh size.
If $F$ is analytic and absolutely integrable
on a strip complex region
\begin{equation}
\domD_d=\{\zeta\in\mathbb{C} : \abs{\Im\zeta}< d\}\quad (d>0),
\label{eq:domDd}
\end{equation}
this approximation is capable of achieving exponential convergence,
denoted as $\OO(\ee^{-\pi d/h})$.
Furthermore, if $\abs{F(x)}$ decays exponentially as $x\to\pm\infty$,
i.e., $\abs{F(x)}=\OO(\ee^{-\mu\abs{x}})$ $(\mu>0)$,
the infinite sum in~\eqref{eq:sinc-approx-inf} may
be truncated at some truncation numbers $M$ and $N$ as
\begin{equation}
\sum_{k=-\infty}^{\infty}
 F(kh) \sinc\left(\frac{x-kh}{h}\right)
\approx \sum_{k=-M}^{N}
 F(kh) \sinc\left(\frac{x-kh}{h}\right).
\label{eq:sinc-approx-finite}
\end{equation}
Its error rate is estimated as $\OO(\ee^{-\mu n h} / h)$,
where $n = \max\{M, N\}$.
The overall error is estimated by
the sum of the error in~\eqref{eq:sinc-approx-inf}
(termed the discretization error) and
that in~\eqref{eq:sinc-approx-finite} (termed the truncation error).
Thus,
by setting the exponential rates of the two errors equal to each other,
we can derive the optimal formula for
the mesh size $h$ with respect to $n$
as
\[
 \ee^{-\pi d /h} = \ee^{-\mu n h}
\quad \Leftrightarrow \quad
h = \sqrt{\frac{\pi d}{\mu n}},
\]
resulting in $\OO(\sqrt{n}\ee^{-\sqrt{\pi d \mu n}})$
as the final (overall) error rate.
The Sinc approximation for
exponentially decaying functions is called
the SE-Sinc approximation,
which has been extensively developed
and analyzed by Stenger~\cite{Stenger,stenger00:_summar}.

Recently, it has been reported~\cite{SugiMatsu,TanaSugiMuro} that
if $\abs{F(x)}$ decays double exponentially as $x\to\pm\infty$,
i.e., $\abs{F(x)}=\OO(\ee^{-(\pi/2)\mu \exp(\abs{x})})$,
the approximation error in~\eqref{eq:sinc-approx-finite}
is $\OO(\ee^{-(\pi/2)\mu\exp(nh)})$,
which is considerably smaller than that
in the case where $\abs{F(x)}$ decays single exponentially.
The Sinc approximation for
double-exponentially decaying functions is called
the DE-Sinc approximation.
In this case, the optimal formula of
the mesh size $h$ with respect to $n$
is given by the equation
\begin{equation}
 \ee^{-\pi d /h} = \ee^{-(\pi/2)\mu\exp(nh)},
\label{eq:double-expo-case}
\end{equation}
but a closed form for $h$
in terms of elementary functions of $n$
cannot be obtained.
Instead, several authors have employed the following near-optimal formula
\begin{equation}
 h = \frac{\log(2 d n/\mu)}{n}.
\label{eq:standard-h-DE}
\end{equation}
By the formula, the left-hand side in~\eqref{eq:double-expo-case}
can be written as $\ee^{-\pi d n/\log(2 d n/\mu)}$,
and the right-hand side in~\eqref{eq:double-expo-case}
can be written as $\ee^{-\pi d n}$.
Although these two rates are approximately equal,
the former is slightly lower.
Consequently, the final error rate of the DE-Sinc approximation
is $\OO(\ee^{-\pi d n/\log(2 d n/\mu)})$.
We note that not only its error rate
but also its computable error bound has been provided~\cite{OkaMatsuSugi,Oka1},
which is useful for a computation with guaranteed accuracy.

Considering the above background, this study improves the formula
that relates $h$ and $n$ in the case of the DE-Sinc approximation.
We propose two improved formulas.
The first one is based on the concept by an earlier research~\cite{OkaKuro},
which focused on the double-exponential formula.
Their improved formula between $h$ and $n$
reduced the gap between the discretization and truncation errors.
As a result,
a sharper error bound for the double-exponential formula
was obtained than the existing one.
Thus, utilizing the same idea,
we provide a sharper error bound for the DE-Sinc approximation
than the existing one~\cite{OkaMatsuSugi,Oka1}.

Our second formula is based on a distinct approach.
Thus far, $n$ has denoted the maximum number of truncation numbers $M$ and $N$,
i.e., $n=\max\{M, N\}$.
This condition has been a barrier to the construction
of the optimal selection formula that
eliminates the gap between the discretization and truncation errors.
In the new formula, we regard $n$ as a free parameter
that does not have to be equal to $M$ or $N$.
Then, we can derive formulas for $h$, $M$, and $N$
with respect to $n$, such that
the discretization and truncation errors
have exactly the same rate.
In consequence, we give a much better error bound
for the DE-Sinc approximation.
Numerical experiments also support our theoretical results.

The remaining sections of this paper are structured as follows.
In Section~\ref{sec2},
after reviewing an existing theorem,
we state two new theorems, which are the primary results of this study.
In Section~\ref{sec3},
we show some numerical examples.
In Section~\ref{sec4},
we provide proofs of the main results.

\section{Existing and new results}\label{sec2}

In this paper,
we consider functions $F(\zeta)$
that are analytic in $\domD_d$ (recall~\eqref{eq:domDd})
and decay double-exponentially as $\Re\zeta\to\pm\infty$.
More specifically,
we consider the following function space.

\begin{definition}
Let $L$, $R$, $\alpha$, and $\beta$
be positive constants,
and let $d$ be a constant with $0<d<\pi/2$.
Then, $\LC_{L,R,\alpha,\beta}(\domD_d)$ denotes
a family of functions $F$ that are analytic in $\domD_d$,
and satisfy
\begin{align}
 \abs{F(\zeta)}&\leq
\frac{L}{\abs{1+\ee^{-\pi\sinh\zeta}}^{\alpha}\abs{1+\ee^{\pi\sinh\zeta}}^{\beta}}\nonumber
\intertext{for all $\zeta\in\domD_d$, and}
 \abs{F(x)}&\leq
\frac{R}{(1+\ee^{-\pi\sinh x})^{\alpha}(1+\ee^{\pi\sinh x})^{\beta}}
\label{eq:double-exp-decay}
\end{align}
for all $x\in\mathbb{R}$.
\end{definition}

The existing error bound
for the DE-Sinc approximation is written as follows.

\begin{theorem}[Okayama~{\cite[Theorem 4.8]{Oka1}}]
\label{thm:exist}
Assume that $F\in\LC_{L,R,\alpha,\beta}(\domD_d)$.
Let $\mu=\min\{\alpha,\beta\}$,
let $\nu=\max\{\alpha,\beta\}$,
let $h$ be defined as in~\eqref{eq:standard-h-DE},
and let $M$ and $N$ be defined as
\begin{align}
\left\{
\begin{aligned}
  M&=n, &N&= n - \left\lfloor\frac{1}{h}\log\left(\frac{\beta}{\alpha}\right)\right\rfloor&&
(\text{if}\,\,\,\mu=\alpha),\\
  N&=n, &M&= n - \left\lfloor\frac{1}{h}\log\left(\frac{\alpha}{\beta}\right)\right\rfloor&&
(\text{if}\,\,\,\mu=\beta).
\end{aligned}
\right.
\label{eq:standard-M-N}
\end{align}
Furthermore, let $n$ be taken sufficiently large so that
$n\geq (\nu\ee)/(2d)$ holds.
Then, it holds that
\[
 \sup_{x\in\mathbb{R}}
\left\lvert
F(x) - \sum_{k=-M}^N F(kh)\sinc\left(\frac{x - kh}{h}\right)
\right\rvert
\leq C \ee^{-\pi d n/\log(2 d n/\mu)},
\]
where $C$ is a constant independent of $n$, expressed as
\[
 C
=\frac{2}{\pi d}
\left[
\frac{2L}{\pi\mu(1 - \ee^{-\pi\mu\ee})\cos^{\alpha+\beta}((\pi/2)\sin d)\cos d}
+ R\ee^{\pi\nu/2}
\right].
\]
\end{theorem}

The above theorem states not only the convergence rate
$\OO(\ee^{-\pi d n/\log(2 d n/\mu)})$,
but also the computable error bound by revealing
the explicit form of the constant $C$.

In the same form,
the first main result of this paper is written as follows.
The proof is given in Section~\ref{sec:proof1}.
For convenience,
we introduce functions $p(x)$ and $q(x)$ as
\begin{align}
 p(x) &= \frac{x}{\arsinh(x / \arsinh x)}, \label{eq:p-x}\\
 q(x) &= \frac{x}{\arsinh x}. \label{eq:q-x}
\end{align}

\begin{theorem}
\label{thm:new1}
Assume that $F\in\LC_{L,R,\alpha,\beta}(\domD_d)$.
Let $\mu=\min\{\alpha,\beta\}$,
let $h$ be defined as
\begin{equation}
h = \frac{\arsinh(q(dn/\mu))}{n},
\label{eq:improve-h-DE-1}
\end{equation}
and let $M$ and $N$ be defined as
\begin{align}
\left\{
\begin{aligned}
  M&=n, &N&= \left\lceil\frac{1}{h}\arsinh\left(\frac{\alpha}{\beta}q\left(\frac{dn}{\mu}\right)\right)\right\rceil&&
(\text{if}\,\,\,\mu=\alpha),\\
  N&=n, &M&= \left\lceil\frac{1}{h}\arsinh\left(\frac{\beta}{\alpha}q\left(\frac{dn}{\mu}\right)\right)\right\rceil&&
(\text{if}\,\,\,\mu=\beta).
\end{aligned}
\right.
\label{eq:improve-M-N-1}
\end{align}
Then, it holds that
\[
 \sup_{x\in\mathbb{R}}
\left\lvert
F(x) - \sum_{k=-M}^N F(kh)\sinc\left(\frac{x - kh}{h}\right)
\right\rvert
\leq C \ee^{-\pi d n/\arsinh(d n/\mu)},
\]
where $C$ is a constant independent of $n$, expressed as
\[
 C = \frac{2}{\pi d}
\left[
\frac{2L\ee^{-\pi\mu(p(d/\mu)-q(d/\mu))}}{\pi\mu(1 - \ee^{-2\pi\mu p(d/\mu)})\cos^{\alpha+\beta}((\pi/2)\sin d)\cos d}
+ \frac{\pi}{2}R
\right].
\]
\end{theorem}

The convergence rate given by Theorem~\ref{thm:new1} is
$\OO(\ee^{-\pi d n/\arsinh(dn/\mu)})$,
which is asymptotically equal to the previous rate:
$\OO(\ee^{-\pi d n/\log(2dn/\mu)})$
(note that
$\arsinh x=\log(x + \sqrt{1 + x^2})\simeq \log (2x)$ as $x\to\infty$).
However, the constant $C$ in Theorem~\ref{thm:new1} may be smaller than
that in Theorem~\ref{thm:exist},
which is also confirmed in numerical experiments provided in the next section.
Furthermore, in Theorem~\ref{thm:exist},
there is an artificial condition on $n$:
$n\geq (\nu\ee)/(2d)$,
which is eliminated in Theorem~\ref{thm:new1}.

In both Theorems~\ref{thm:exist} and~\ref{thm:new1},
the integer $n$ satisfies
$n = \max\{M, N\}$,
i.e., either $M$ or $N$ equals to $n$.
In the next theorem,
both $M$ and $N$ may be taken smaller than $n$.
This theorem is the second main result of this paper.
The proof is given in Section~\ref{sec:proof2}.

\begin{theorem}
\label{thm:new2}
Assume that $F\in\LC_{L,R,\alpha,\beta}(\domD_d)$.
Let $\mu=\min\{\alpha,\beta\}$,
let $h$ be defined as
\begin{equation}
h = \frac{\arsinh(dn/\mu)}{n},
\label{eq:improve-h-DE-2}
\end{equation}
and let $M$ and $N$ be defined as
\begin{align}
\begin{aligned}
  M&=\left\lceil\frac{1}{h}\arsinh\left(\frac{\mu}{\alpha}q\left(\frac{dn}{\mu}\right)\right)\right\rceil,
 &N&= \left\lceil\frac{1}{h}\arsinh\left(\frac{\mu}{\beta}q\left(\frac{dn}{\mu}\right)\right)\right\rceil.
\end{aligned}
\label{eq:improve-M-N-2}
\end{align}
Then, it holds that
\[
 \sup_{x\in\mathbb{R}}
\left\lvert
F(x) - \sum_{k=-M}^N F(kh)\sinc\left(\frac{x - kh}{h}\right)
\right\rvert
\leq C \ee^{-\pi d n/\arsinh(d n/\mu)},
\]
where $C$ is a constant independent of $n$, expressed as
\begin{equation}
C= \frac{2}{\pi d}
\left[
\frac{2L}{\pi\mu(1 - \ee^{-2\pi\mu q(d/\mu)})\cos^{\alpha+\beta}((\pi/2)\sin d)\cos d}
+ R
\right].
\label{eq:constant-C-thm:new2}
\end{equation}
\end{theorem}

For simplicity, let us consider the case where $\alpha=\beta$.
In this case, from both formulas~\eqref{eq:standard-M-N}
and~\eqref{eq:improve-M-N-1},
substituting the formula of $h$,
we obtain $M=N=n$.
In contrast, from the formula~\eqref{eq:improve-M-N-2},
substituting the formula of $h$,
we obtain
\[
 M=N=\left\lceil
n\cdot \frac{\arsinh(q(dn/\mu))}{\arsinh(dn/\mu)}
\right\rceil.
\]
Because $q(x)< x$ for $x> \sinh 1$,
this equation implies that $M=N < n$ for sufficiently large $n$.
Although the convergence rates
given by Theorems~\ref{thm:new1} and~\ref{thm:new2}
seem exactly the same,
we may obtain a better result in the case of Theorem~\ref{thm:new2},
because the number of function evaluations
$M+N+1$ may be smaller than that in the case of Theorem~\ref{thm:new1}.
This is also confirmed in numerical experiments provided in the next section.

\section{Numerical experiments}
\label{sec3}

In this section, we present numerical results
using the formulas described
in Theorems~\ref{thm:exist}--\ref{thm:new2}.
We implemented the programs in C
with double-precision arithmetic.

The numerical experiments are conducted on the following two functions:
\begin{align}
 f_1(t) &= (1 - t^2)^{1/2}, \label{eq:f_1}\\
 f_2(t) &= (1 + t^2)^{1/2}(1 + t)^{1/2}(1 - t)^{3/4}, \label{eq:f_2}
\end{align}
which are defined on $[-1, 1]$.
The second one was also conducted in
Okayama et al.~\cite[Example 3.1]{OkaMatsuSugi}.
Before applying the Sinc approximation,
we employ the double-exponential transformation given by
\[
 t = \phi(x) = \tanh\left(\frac{\pi}{2}\sinh x\right).
\]
That is, putting $F_i(x)=f_i(\phi(x))$, we obtain the
DE-Sinc approximation formula
\[
 F_i(x) \approx\sum_{k=-M}^N F_i(kh)\sinc\left(\frac{x - kh}{h}\right).
\]
The function $F_1$ satisfies
the assumptions in Theorems~\ref{thm:exist}--\ref{thm:new2}
with $L=R=2$, $\alpha=\beta=1/2$, and $d=3/2$.
The function $F_2$
satisfies
the assumptions in Theorems~\ref{thm:exist}--\ref{thm:new2}
with $L=R=4$, $\alpha=1/2$, $\beta=3/4$, and $d=\pi/6$.
We examined its approximation error
on the following 20001 points:
\[
 x = \frac{4i}{10000}\quad
 (i = -10000,\,-9999,\,\ldots,\,-1,\,0,\,1,\,\ldots,\,9999,\,10000),
\]
and investigated the maximum error among these points.
The results are shown in Figs.~\ref{fig:1} and~\ref{fig:2}.
In both figures,
the error bounds given in Theorems~\ref{thm:exist}--\ref{thm:new2}
include corresponding observed errors.
Furthermore, the figures also demonstrate that
the constant $C$ in Theorem~\ref{thm:new1}
is improved compared to that in Theorem~\ref{thm:exist},
although the convergence rate remains unchanged.
In contrast,
the convergence rate by Theorem~\ref{thm:new2}
is higher than that by Theorems~\ref{thm:exist} and~\ref{thm:new1}.
This reflects the fact that $M+N+1$ (the number of function evaluations)
in Theorem~\ref{thm:new2} is smaller than
that in Theorems~\ref{thm:exist} and~\ref{thm:new1}
(note that the horizontal axis is not $n$, but $M+N+1$).

\begin{figure}[htbp]
\centering
\includegraphics[scale=.9]{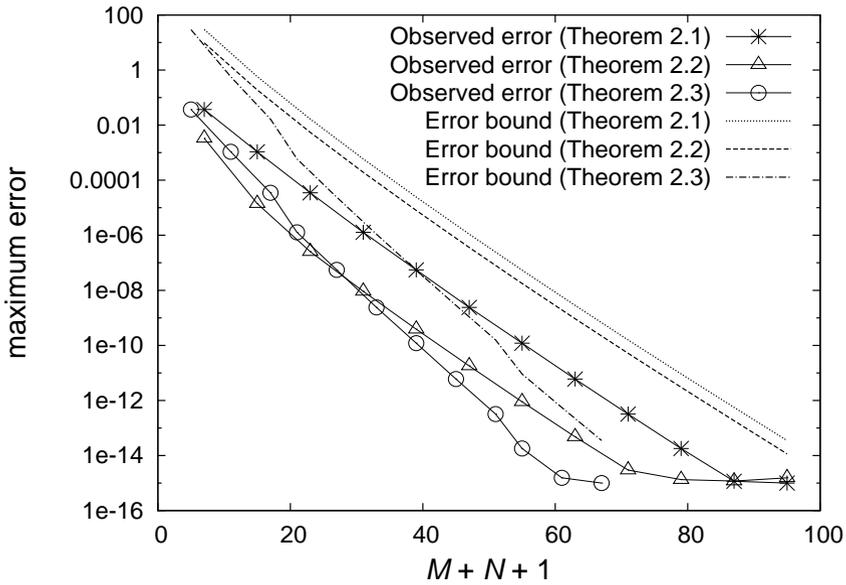}
\caption{Observed error and error bound for the function $f_1$ in~\eqref{eq:f_1}.}
\label{fig:1}
\end{figure}
\begin{figure}[htbp]
\centering
\includegraphics[scale=.9]{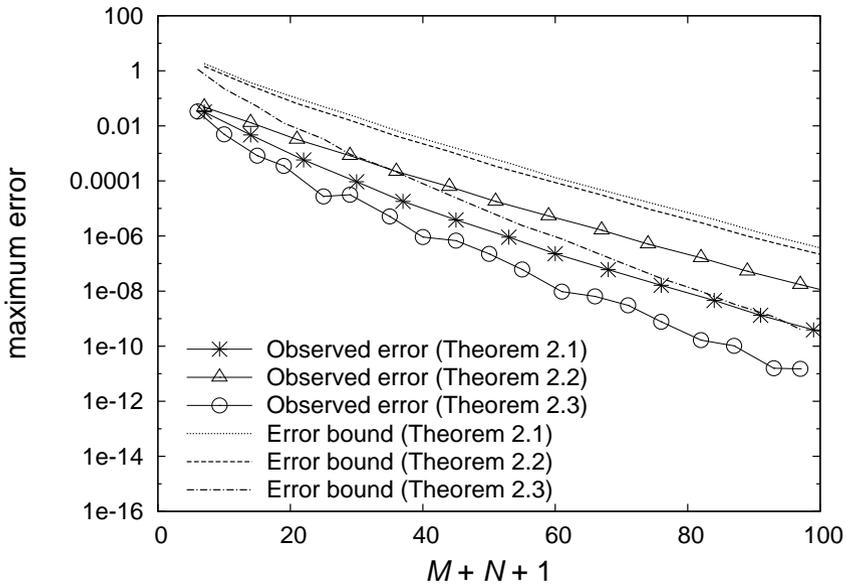}
\caption{Observed error and error bound for the function $f_2$ in~\eqref{eq:f_2}.}
\label{fig:2}
\end{figure}

\section{Proofs}
\label{sec4}

\subsection{Sketch of the proofs}

In this section, we provide the proofs
for Theorems~\ref{thm:new1} and~\ref{thm:new2}.
In both cases,
as outlined in Section~\ref{sec1},
we divide the approximation error into
the sum of the
discretization and truncation errors.
More precisely, we estimate the error as
\begin{align*}
& \left\lvert
  F(x) - \sum_{k=-M}^N F(kh)\sinc\left(\frac{x - kh}{h}\right)
 \right\rvert\\
&\leq
 \left\lvert
  F(x) - \sum_{k=-\infty}^{\infty} F(kh)\sinc\left(\frac{x - kh}{h}\right)
 \right\rvert\\
&\quad+\left\lvert
\sum_{k=-\infty}^{-M-1}F(kh)\sinc\left(\frac{x - kh}{h}\right)
+
\sum_{k=N+1}^{\infty}F(kh)\sinc\left(\frac{x - kh}{h}\right)
 \right\rvert.
\end{align*}
As for the first term (the discretization error),
we use the following bound.

\begin{theorem}[Okayama~{\cite[Part of Theorem 4.8]{Oka1}}]
\label{thm:discrete}
Assume that $F\in\LC_{L,R,\alpha,\beta}(\domD_d)$.
Then, it holds that
\begin{align*}
& \sup_{x\in\mathbb{R}}\left\lvert
F(x) - \sum_{k=-\infty}^{\infty}F(kh)\sinc\left(\frac{x - kh}{h}\right)
\right\rvert\\
&\leq \frac{4L}{\pi^2 d \mu(1 - \ee^{-2\pi d/h})\cos^{\alpha+\beta}((\pi/2)\sin d)\cos d}\ee^{-\pi d/h},
\end{align*}
where $\mu=\min\{\alpha,\beta\}$.
\end{theorem}

As for the second term (the truncation error),
we give the following bound.
The proof is provided in Section~\ref{lem:new}.

\begin{lemma}
\label{lem:new}
Assume that $F\in\LC_{L,R,\alpha,\beta}(\domD_d)$.
Let $\mu=\min\{\alpha,\beta\}$,
and let $M$ and $N$ be defined as in~\eqref{eq:improve-M-N-2}.
Then, it holds that
\begin{align*}
&\sup_{x\in\mathbb{R}}\left\lvert
\sum_{k=-\infty}^{-M-1}F(kh)\sinc\left(\frac{x - kh}{h}\right)
+
\sum_{k=N+1}^{\infty}F(kh)\sinc\left(\frac{x - kh}{h}\right)
 \right\rvert\\
&\leq \frac{2R}{\pi\mu h\sqrt{1 + \{q(dn/\mu)\}^2}}\ee^{-\pi\mu q(dn/\mu)}.
\end{align*}
\end{lemma}

We then specify the relation between $h$ and $n$,
considering the exponential rates of the two errors:
$\ee^{-\pi d/h}$ and $\ee^{-\pi\mu q(dn/\mu)}$.
In Theorem~\ref{thm:new1},
we set the mesh size $h$ as in~\eqref{eq:improve-h-DE-1}.
In this case, the formula~\eqref{eq:improve-M-N-2}
is rewritten as~\eqref{eq:improve-M-N-1},
from which $n = \max\{M, N\}$ holds.
Furthermore,
we have
\begin{align*}
 (\text{the discretization error}) &= \OO(\ee^{-\pi d n/\arsinh(q(dn/\mu))}),\\
 (\text{the truncation error}) &= \OO(\ee^{-\pi d n/\arsinh(dn/\mu)}),
\end{align*}
and the latter order is slightly worse than the former.
Therefore, the convergence rate of the overall error is given as
$\OO(\ee^{-\pi d n/\arsinh(dn/\mu)})$.
The detailed proof is provided in Section~\ref{sec:proof1}.

In Theorem~\ref{thm:new2},
we set the mesh size $h$ as in~\eqref{eq:improve-h-DE-2}.
In this case, although $n=\max\{M, N\}$ does not hold,
$\ee^{-\pi d/h}=\ee^{-\pi\mu q(dn/\mu)}$ hold,
i.e., we have
\begin{align*}
 (\text{the discretization error}) &= \OO(\ee^{-\pi d n/\arsinh(dn/\mu)}),\\
 (\text{the truncation error}) &= \OO(\ee^{-\pi d n/\arsinh(dn/\mu)}),
\end{align*}
and both orders are exactly the same.
Thus, the convergence rate of the overall error is given as
$\OO(\ee^{-\pi d n/\arsinh(dn/\mu)})$.
The detailed proof is provided in Section~\ref{sec:proof2}.

\subsection{Proof of Lemma~\ref{lem:new}}

We show Lemma~\ref{lem:new} as follows.
\begin{proof}
First, using $\abs{\sinc x}\leq 1$, we have
\begin{align*}
&\left\lvert
\sum_{k=-\infty}^{-M-1}F(kh)\sinc\left(\frac{x - kh}{h}\right)
+
\sum_{k=N+1}^{\infty}F(kh)\sinc\left(\frac{x - kh}{h}\right)
 \right\rvert\\
&\leq \sum_{k=-\infty}^{-M-1}\abs{F(kh)} + \sum_{k=N+1}^{\infty}\abs{F(kh)}.
\end{align*}
Using~\eqref{eq:double-exp-decay}, we estimate the second term as
\begin{align*}
 \sum_{k=N+1}^{\infty}\abs{F(kh)}
&\leq R\sum_{k=N+1}^{\infty}
\frac{1}{(1+\ee^{-\pi\sinh(kh)})^{\alpha}(1+\ee^{\pi\sinh(kh)})^{\beta}}\\
&=R\sum_{k=N+1}^{\infty}
\frac{\ee^{-\pi\beta\sinh(kh)}}{(1+\ee^{-\pi\sinh(kh)})^{\alpha+\beta}}\\
&\leq R\sum_{k=N+1}^{\infty}
\frac{\ee^{-\pi\beta\sinh(kh)}}{(1 + 0)^{\alpha+\beta}}\\
&\leq \frac{R}{h}\int_{Nh}^{\infty}\ee^{-\pi\beta\sinh x}\dd x\\
&\leq \frac{R}{h\pi\beta \cosh(Nh)}
\int_{Nh}^{\infty}\pi\beta\cosh(x)\ee^{-\pi\beta\sinh x}\dd x\\
&= \frac{R}{h\pi\beta\cosh(Nh)}\ee^{-\pi\beta\sinh(Nh)}.
\end{align*}
Furthermore, using~\eqref{eq:improve-M-N-2}, we have
\begin{align*}
& \frac{R}{h\pi\beta\cosh(Nh)}\ee^{-\pi\beta\sinh(Nh)}\\
&\leq \frac{R}{h\pi\beta\cosh(\arsinh((\mu/\beta)q(dn/\mu)))}
\ee^{-\pi\beta\sinh(\arsinh((\mu/\beta)q(dn/\mu)))}\\
&=\frac{R}{h\pi\beta\sqrt{1 + \{(\mu/\beta)q(dn/\mu)\}^2}}
\ee^{-\pi\mu q(dn/\mu)}\\
&=\frac{R}{h\pi\sqrt{\beta^2 + \mu^2\{q(dn/\mu)\}^2}}
\ee^{-\pi\mu q(dn/\mu)}\\
&\leq\frac{R}{h\pi\sqrt{\mu^2 + \mu^2\{q(dn/\mu)\}^2}}
\ee^{-\pi\mu q(dn/\mu)}\\
&=\frac{R}{h\pi\mu\sqrt{1 + \{q(dn/\mu)\}^2}}
\ee^{-\pi\mu q(dn/\mu)}.
\end{align*}
In the same way,
using~\eqref{eq:double-exp-decay} and~\eqref{eq:improve-M-N-2},
we estimate the first term as
\[
 \sum_{k=-\infty}^{-M-1}\abs{F(kh)}
\leq \frac{R}{h\pi\mu\sqrt{1 + \{q(dn/\mu)\}^2}}
\ee^{-\pi\mu q(dn/\mu)}.
\]
Thus, we obtain the desired inequality.
\end{proof}

\begin{remark}
As seen in this proof,
the formula~\eqref{eq:improve-M-N-2} is designed so that
$\ee^{-\pi\alpha\sinh(Mh)}\leq \ee^{-\pi\mu q(dn/\mu)}$
and
$\ee^{-\pi\beta\sinh(Nh)}\leq \ee^{-\pi\mu q(dn/\mu)}$
hold.
We note that the explicit form of the function $q(x)$ is not used
in this proof;
this point is discussed in Section~\ref{sec:remark-q}.
\end{remark}

\subsection{Proof of Theorem~\ref{thm:new1}}
\label{sec:proof1}

For the proof of Theorem~\ref{thm:new1},
we use the following two propositions.

\begin{proposition}[Okayama and Kurogi~{\cite[Proposition 8]{OkaKuro}}]
\label{prop:p-q-r}
Let $p(x)$ and $q(x)$ be defined by~\eqref{eq:p-x}
and~\eqref{eq:q-x}, respectively.
Then, $p(x)$ and $r(x) = p(x) - q(x)$
are monotonically increasing functions for $x\geq 0$.
\end{proposition}
\begin{proposition}
\label{prop:sinh}
For all $t\geq 0$, we have
\begin{equation}
 \sinh t\geq t \sinh\left(\frac{2}{\pi}t\right).
\label{eq:sinh}
\end{equation}
\end{proposition}
\begin{proof}
Setting $f(t)=\sinh(t) - t\sinh((2/\pi) t)$,
we show $f(t)\geq 0$ for $t\geq 0$, from which we obtain the conclusion.
Using
\begin{align*}
&\sinh t\\
&=\sinh\left(\frac{2}{\pi}t + \left(1 - \frac{2}{\pi}\right)t\right)\\
&=\sinh\left(\frac{2}{\pi}t\right)
\cosh\left(\left(1 - \frac{2}{\pi}\right)t\right)
+\cosh\left(\frac{2}{\pi}t\right)
\sinh\left(\left(1 - \frac{2}{\pi}\right)t\right)\\
&=\sinh\left(\frac{2}{\pi}t\right)
\cosh\left(\left(1 - \frac{2}{\pi}\right)t\right)
+\left\{\sinh\left(\frac{2}{\pi}t\right) + \exp\left(-\frac{2}{\pi}t\right)\right\}
\sinh\left(\left(1 - \frac{2}{\pi}\right)t\right)\\
&=\sinh\left(\frac{2}{\pi}t\right)
\exp\left(\left(1 - \frac{2}{\pi}\right)t\right)
+\exp\left(-\frac{2}{\pi}t\right)
\sinh\left(\left(1 - \frac{2}{\pi}\right)t\right),
\end{align*}
we have
\begin{align*}
 f(t) &= g(t)\sinh\left(\frac{2}{\pi}t\right) + h(t),
\end{align*}
where
\begin{align*}
 g(t) &= \exp\left(\left(1 - \frac{2}{\pi}\right)t\right) - t,\\
 h(t) &= \exp\left(-\frac{2}{\pi}t\right)
\sinh\left(\left(1 - \frac{2}{\pi}\right)t\right).
\end{align*}
Here, we set $\gamma$ and $\delta$ as
\begin{align*}
 \gamma &= \frac{\pi}{\pi - 2}\log\left(\frac{\pi}{\pi - 2}\right),\\
 \delta &= \frac{\pi}{2(\pi - 2)}\log\left(\frac{\pi}{4 - \pi}\right),
\end{align*}
which satisfy $g'(\gamma)=0$ and $h'(\delta)=0$.
In what follows, we consider three cases:
(i) $0\leq t < \pi \log 2$,
(ii) $\pi\log 2\leq t < \pi\log 3$, and
(iii) $\pi \log 3 \leq t$, and we show $f(t)\geq 0$ in each case.

(i) Consider the case $0\leq t < \pi \log 2$.
Because $g(t)$ is monotonically decreasing
for $t\in (0, \gamma)$,
noting $\pi\log 2< \gamma$,
we have
\[
 f(t) \geq g(\pi\log 2)\sinh\left(\frac{2}{\pi}t\right) + h(t).
\]
Furthermore, from $g(\pi\log 2)>0$,
$\sinh((2/\pi)t)\geq 0$ and $h(t)\geq 0$, we have $f(t)\geq 0$.

(ii) Consider the case $\pi\log 2\leq t < \pi\log 3$.
Because $g(t)$ has its minimum at $t = \gamma$,
noting $\gamma\in (\pi\log2, \pi\log 3)$,
we have
\[
 f(t)\geq g(\gamma)\sinh\left(\frac{2}{\pi}t\right) + h(t).
\]
Because $g(\gamma) < 0$,
$g(\gamma)\sinh((2/\pi)t)$ is a monotonically decreasing function.
Furthermore, because $h(t)$ is monotonically decreasing for $t\geq \delta$,
noting $\delta < \pi\log 2$, we have
\[
 g(\gamma)\sinh\left(\frac{2}{\pi}t\right) + h(t)
\geq g(\gamma)\sinh\left(\frac{2}{\pi}\cdot\pi\log 3\right) + h(\pi\log 3)
> 0.
\]

(iii) Consider the case $\pi \log 3 \leq t$.
Because $g(t)$ is monotonically increasing
for $t\geq \gamma$,
noting $\gamma < \pi\log 3$, we have
\[
 f(t)\geq g(\pi\log 3)\sinh\left(\frac{2}{\pi}t\right) + h(t).
\]
Furthermore, because $g(\pi\log 3)>0$,
$\sinh((2/\pi) t)>0$, and $h(t)>0$,
we have $f(t)>0$.
This completes the proof.
\end{proof}

Using these propositions,
we prove Theorem~\ref{thm:new1} as follows.

\begin{proof}
From Theorem~\ref{thm:discrete} and Lemma~\ref{lem:new},
substituting~\eqref{eq:improve-h-DE-1},
we have
\begin{align}
& \sup_{x\in\mathbb{R}}\left\lvert
F(x) - \sum_{k=-M}^N F(kh)\sinc\left(\frac{x - kh}{h}\right)
\right\rvert\nonumber\\
&\leq\frac{4L}{\pi^2 d \mu\cos^{\alpha+\beta}((\pi/2)\sin d)\cos d}
\cdot\frac{\ee^{-\pi d/h}}{(1 - \ee^{-2\pi d/h})}\nonumber\\
&\quad + \frac{2 R}{\pi \mu h\sqrt{1 + \{q(dn/\mu)\}^2}}
\ee^{-\pi \mu q(dn/\mu)}\nonumber\\
&=\frac{4L}{\pi^2 d \mu\cos^{\alpha+\beta}((\pi/2)\sin d)\cos d}
\cdot\frac{\ee^{-\pi\mu p(dn/\mu)}}{(1 - \ee^{-2\pi\mu p(dn/\mu)})}\nonumber\\
&\quad+ \frac{2 R (dn/\mu)}{\pi d\arsinh(q(dn/\mu))\sqrt{1 + \{q(dn/\mu)\}^2}}
\ee^{-\pi\mu q(dn/\mu)}.
\label{eq:bound-disc-trunc}
\end{align}
Here, for $t>0$, the inequality~\eqref{eq:sinh} is rewritten as
\begin{align*}
 \frac{\sinh t}{t}&\geq \sinh\left(\frac{2}{\pi}t\right)\\
\Leftrightarrow\quad
\arsinh\left(\frac{\sinh t}{t}\right)&\geq \frac{2}{\pi} t\\
\Leftrightarrow\quad
\frac{t}{\arsinh\left(\frac{\sinh t}{t}\right)}
&\leq \frac{\pi}{2}.
\end{align*}
Putting $t=\arsinh(dn/\mu)$ in the last inequality, we have
\[
 \frac{\arsinh(dn/\mu)}{\arsinh\left(\frac{dn/\mu}{\arsinh(dn/\mu)}\right)}
=\frac{\arsinh(dn/\mu)}{\arsinh(q(dn/\mu))}\leq \frac{\pi}{2}.
\]
Therefore, the second term in~\eqref{eq:bound-disc-trunc}
is bounded as
\begin{align*}
&\frac{2 R (dn/\mu)}{\pi d\arsinh(q(dn/\mu))\sqrt{1 + \{q(dn/\mu)\}^2}}
\ee^{-\pi\mu q(dn/\mu)}\\
&=\frac{2R}{\pi d}\cdot
\frac{\arsinh(dn/\mu)}{\arsinh(q(dn/\mu))}\cdot
\frac{q(dn/\mu)}{\sqrt{1 + \{q(dn/\mu)\}^2}}
\ee^{-\pi\mu q(dn/\mu)}\\
&\leq \frac{2R}{\pi d}\cdot\frac{\pi}{2}\cdot
\frac{q(dn/\mu)}{\sqrt{0 + \{q(dn/\mu)\}^2}}\ee^{-\pi\mu q(dn/\mu)}\\
&=\frac{R}{d}\ee^{-\pi\mu q(dn/\mu)}.
\end{align*}
Furthermore, from Proposition~\ref{prop:p-q-r},
the first term in~\eqref{eq:bound-disc-trunc}
is bounded as
\begin{align*}
&\frac{4L}{\pi^2 d \mu\cos^{\alpha+\beta}((\pi/2)\sin d)\cos d}
\cdot\frac{\ee^{-\pi\mu p(dn/\mu)}}{(1 - \ee^{-2\pi\mu p(dn/\mu)})}\\
&=\frac{4L}{\pi^2 d \mu\cos^{\alpha+\beta}((\pi/2)\sin d)\cos d}
\cdot
\frac{\ee^{-\pi\mu \{p(dn/\mu) - q(dn/\mu)\}}}{(1 - \ee^{-2\pi\mu p(dn/\mu)})}
\ee^{-\pi \mu q(dn/\mu)}\\
&\leq\frac{4L}{\pi^2 d \mu\cos^{\alpha+\beta}((\pi/2)\sin d)\cos d}
\cdot
\frac{\ee^{-\pi\mu \{p(d/\mu) - q(d/\mu)\}}}{(1 - \ee^{-2\pi\mu p(d/\mu)})}
\ee^{-\pi \mu q(dn/\mu)}.
\end{align*}
This completes the proof.
\end{proof}

\subsection{Proof of Theorem~\ref{thm:new2}}
\label{sec:proof2}

For the proof of Theorem~\ref{thm:new2},
we use the following proposition.

\begin{proposition}[Okayama and Kawai~{\cite[Proposition 7]{OkaKawai}}]
\label{prop:q-x}
Let $q(x)$ be defined by~\eqref{eq:q-x}.
Then, $q(x)$
is a monotonically increasing function for $x\geq 0$.
\end{proposition}

Using this result, we prove Theorem~\ref{thm:new2} as follows.
\begin{proof}
From Theorem~\ref{thm:discrete} and Lemma~\ref{lem:new},
substituting~\eqref{eq:improve-h-DE-2},
we have
\begin{align}
&\sup_{x\in\mathbb{R}}\left\lvert
F(x) - \sum_{k=-M}^N F(kh)\sinc\left(\frac{x - kh}{h}\right)
\right\rvert\nonumber\\
&\leq\frac{4L}{\pi^2 d \mu\cos^{\alpha+\beta}((\pi/2)\sin d)\cos d}
\cdot\frac{\ee^{-\pi d/h}}{(1 - \ee^{-2\pi d/h})}\nonumber\\
&\quad + \frac{2 R}{\pi \mu h\sqrt{1 + \{q(dn/\mu)\}^2}}
\ee^{-\pi \mu q(dn/\mu)}\nonumber\\
&=\frac{4L}{\pi^2 d \mu\cos^{\alpha+\beta}((\pi/2)\sin d)\cos d}
\cdot\frac{\ee^{-\pi\mu q(dn/\mu)}}{(1 - \ee^{-2\pi \mu q(dn/\mu)})}\nonumber\\
&\quad + \frac{2 R q(dn/\mu)}{\pi d\sqrt{1 + \{q(dn/\mu)\}^2}}
\ee^{-\pi \mu q(dn/\mu)}.
\label{eq:bound-disc-trunc-2}
\end{align}
The second term in~\eqref{eq:bound-disc-trunc-2}
is bounded as
\begin{align*}
\frac{2 R q(dn/\mu)}{\pi d\sqrt{1 + \{q(dn/\mu)\}^2}}
\ee^{-\pi \mu q(dn/\mu)}
&=\frac{2R}{\pi d}\cdot\frac{q(dn/\mu)}{\sqrt{1 + \{q(dn/\mu)\}^2}}
\ee^{-\pi\mu q(dn/\mu)}\\
&\leq\frac{2R}{\pi d}\cdot\frac{q(dn/\mu)}{\sqrt{0 + \{q(dn/\mu)\}^2}}
\ee^{-\pi\mu q(dn/\mu)}\\
&=\frac{2R}{\pi d}\ee^{-\pi\mu q(dn/\mu)}.
\end{align*}
Furthermore,
because $q(x)$ monotonically increases for $x\geq 0$,
the first term in~\eqref{eq:bound-disc-trunc-2}
is bounded as
\begin{align*}
&\frac{4L}{\pi^2 d \mu\cos^{\alpha+\beta}((\pi/2)\sin d)\cos d}
\cdot\frac{\ee^{-\pi\mu q(dn/\mu)}}{(1 - \ee^{-2\pi\mu q(dn/\mu)})}\\
&\leq\frac{4L}{\pi^2 d \mu\cos^{\alpha+\beta}((\pi/2)\sin d)\cos d}
\cdot\frac{\ee^{-\pi\mu q(dn/\mu)}}{(1 - \ee^{-2\pi\mu q(d/\mu)})}.
\end{align*}
This completes the proof.
\end{proof}

\subsection{Remark on the function $q(x)$}
\label{sec:remark-q}

So far, we have defined the function $q(x)$ by~\eqref{eq:q-x}.
However, notice that
in the proof of Theorem~\ref{thm:discrete} and Lemma~\ref{lem:new},
we do not use the explicit form of $q(x)$.
Therefore, even if we choose another function for $q(x)$,
we can establish the following theorem.
The proof is omitted here because it can be shown in exactly the same way
as Theorem~\ref{thm:new2}.

\begin{theorem}
\label{thm:new3}
Assume that $F\in\LC_{L,R,\alpha,\beta}(\domD_d)$.
Here, let $q(x)$ be a function that satisfies
$q(x)\geq 0$ for $x\geq 0$ and $q'(x)\geq 0$
for $x\geq 0$.
Let $\mu=\min\{\alpha,\beta\}$,
let $h$ be defined as
\[
h = \frac{d}{\mu q(dn/\mu)},
\]
and let $M$ and $N$ be defined as~\eqref{eq:improve-M-N-2}.
Then, it holds that
\[
 \sup_{x\in\mathbb{R}}
\left\lvert
F(x) - \sum_{k=-M}^N F(kh)\sinc\left(\frac{x - kh}{h}\right)
\right\rvert
\leq C \ee^{-\pi \mu q(dn/\mu)},
\]
where $C$ is a constant independent of $n$,
expressed as~\eqref{eq:constant-C-thm:new2}.
\end{theorem}

According to this theorem,
if we define $q(x)$ as $q(x)=x$,
we can derive the following result instead of Theorem~\ref{thm:new2}.

\begin{corollary}
\label{cor:new3}
Assume that $F\in\LC_{L,R,\alpha,\beta}(\domD_d)$.
Let $\mu=\min\{\alpha,\beta\}$,
let $h$ be defined as
\begin{equation}
h = \frac{1}{n},
\label{eq:improve-h-DE-4}
\end{equation}
and let $M$ and $N$ be defined as
\begin{align}
\begin{aligned}
  M&=\left\lceil\frac{1}{h}\arsinh\left(\frac{dn}{\alpha}\right)\right\rceil,
 &N&= \left\lceil\frac{1}{h}\arsinh\left(\frac{dn}{\beta}\right)\right\rceil.
\end{aligned}
\label{eq:improve-M-N-4}
\end{align}
Then, it holds that
\[
 \sup_{x\in\mathbb{R}}
\left\lvert
F(x) - \sum_{k=-M}^N F(kh)\sinc\left(\frac{x - kh}{h}\right)
\right\rvert
\leq C \ee^{-\pi d n},
\]
where $C$ is a constant independent of $n$, expressed as
\[
C= \frac{2}{\pi d}
\left[
\frac{2L}{\pi\mu(1 - \ee^{-2\pi d})\cos^{\alpha+\beta}((\pi/2)\sin d)\cos d}
+ R
\right].
\]
\end{corollary}

The convergence rate given by this result is
$\OO(\ee^{-\pi d n})$, which appears to be faster than
that by Theorem~\ref{thm:new2}: $\OO(\ee^{-\pi d n/\arsinh(dn/\mu)})$.
However, we emphasize that these rates are given
with respect to $n$,
not with respect to the number of function evaluations ($M+N+1$).
According to~\eqref{eq:improve-h-DE-4}
and~\eqref{eq:improve-M-N-4},
we derive $M=N=\OO(n\log n)$,
which grows more rapidly than $n$.

In this way, by choosing a higher order function for $q(x)$,
we can obtain a higher convergence rate with respect to $n$.
However, in exchange for the higher convergence rate,
we obtain higher growth of $M+N+1$.
As a consequence, we cannot improve the convergence rate
with respect to $M+N+1$.
Although Corollary~\ref{cor:new3} is meaningful
as an error bound with simpler formulas of $h$, $M$, and $N$,
we should note that the convergence rate is given
not with respect to $M+N+1$ but with respect to $n$.

\backmatter

%
%
%

\bmhead{Acknowledgments}
This work was partially supported by the
JSPS Grant-in-Aid for Young Scientists (B)
Number JP17K14147.

%


\bibliography{sn-bibliography}


\end{document}